\documentclass[12pt]{amsart}

\usepackage{graphicx,indentfirst}
\usepackage{amsmath,amssymb,mathrsfs}
\usepackage{amsthm,amscd}
\usepackage{verbatim}
\usepackage{appendix}
\usepackage{enumitem,titletoc}
\usepackage{appendix}
\usepackage{imakeidx}
\makeindex[columns=2, title=Alphabetical Index]
\usepackage{fancyhdr}
\usepackage{amsfonts,color}
\usepackage[all]{xy}
\usepackage{tikz-cd}
\usepackage{syntonly}
\usepackage{fancyhdr}
\newtheorem{theorem}{Theorem}[section]
\newtheorem{lemma}{Lemma}[section]

\newtheorem{conjecture}{Conjecture}[section]
\usepackage[left=2.55cm, right=2.55cm, bottom=3cm, top=3cm]{geometry}

\usepackage{tikz}
\usepackage{extarrows}
\usepackage{hyperref}

\def\XXint#1#2#3{{\setbox0=\hbox{$#1{#2#3}{\int}$ }
\vcenter{\hbox{$#2#3$ }}\kern-.6\wd0}}

\newtheorem{prop}{Proposition}[section]
\allowdisplaybreaks

\newcommand{\ddbar}{i\partial\bar\partial}
\newcommand{\ric}{\mathrm{Ric}}

\newcommand{\eN}{\mathcal{N}}

\makeindex

\DeclareMathOperator{\vol}{Vol}

\DeclareMathOperator{\tr}{tr}

\numberwithin{equation}{section}

\hypersetup{
    colorlinks=true,
    citecolor=green,
    filecolor=green,
    linkcolor=blue,
    urlcolor=black
}

\begin{document}

\address{Department of Mathematics \& Computer Science, Rutgers University, Newark, NJ 07102, United States of America}

\email{bguo@rutgers.edu}

\address{Institute of Mathematics, Academy of Mathematics and Systems Science, Chinese Academy of Sciences, Beijing, 100190, People's Republic of China}
\email{wangjian@amss.ac.cn}

\address{
Department of Mathematics, Nanjing University, Nanjing, 210093, People's Republic of
China }
\email{shiyl@nju.edu.cn}

\address{Department of Mathematics, Rutgers University, Piscataway, NJ 08854, United States of America}
\email{jiansong@math.rutgers.edu}

\title{CscK metrics near the canonical class}  
\author{Bin Guo \and Wangjian Jian\and Yalong Shi\and Jian Song  }\date{}

\begin{abstract} Let $X$ be a K\"ahler manifold with semi-ample canonical bundle $K_X$. It is proved in \cite{JSS} that for any K\"ahler class $\gamma$,  there exists $\delta>0$ such that  for all $t\in (0, \delta)$ there exists a unique cscK metric $g_t$ in $K_X+ t \gamma $.  In this paper, 
we prove that $\{ (X, g_t) \}_{ t\in (0, \delta)} $ have uniformly bounded K\"ahler potentials, volume forms and diameters. As a consequence, these metric spaces are pre-compact in the Gromov-Hausdorff sense.

\end{abstract}

\thanks{B. Guo is supported in part by the National Science Foundation under grant  DMS-2303508 and the collaboration grant 946730 from Simons Foundation. W. Jian is supported in part by the National Natural Science Foundation of China under grants NSFC No.12371058 and No.12288201. Y. Shi is supported in part by NSFC No.12371058. J. Song is supported in part by the National Science Foundation under grant DMS-2203607.}

\maketitle

\section{Introduction}\label{sec:intro}

The existence of constant scalar curvature K\"ahler (cscK) metrics and the related moduli problem are   fundamental problems in complex differential geometry. The works of Chen-Cheng \cite{CC} prove that the cscK metric equation can be solved if the Mabuchi $K$-energy is proper. Such properness of the Mabuchi $K$-energy is closely related to the $J$-equation in the case when the canonical class of the underlying K\"ahler manifold is semi-positive \cite{Ch, We}. In fact, if $X$ is a minimal model, i.e., the canonical bundle $K_X$ is nef, it is proved in \cite{JSS, Zak} that there always exists a unique cscK metric in any K\"ahler class class sufficiently close to $K_X$. Naturally, one would like to establish a compactness result  and to gain further understanding of geometric degeneration for such cscK metrics in relation to the  moduli problem.

Let $X$ be a compact K\"ahler manifold of complex dimension $n$.  Suppose the canonical line bundle $K_X$ is {\em semiample}, i.e., $K_X^m$ is base point free for some $m \geq 1$. For sufficiently large $m\in \mathbb{Z}^+$, the linear system $|mK_X|$ induces a  holomorphic map 
\begin{equation}\label{plmap}
\pi: X\to \mathbb{CP}^N
\end{equation}
 for some $N\in\mathbb N$. $\pi(X)$, the image of $X$ via $\pi$, coincides with the unique algebraic canonical model $X_{can}$ determined by the canonical ring of $X$. The dimension $\kappa =\dim X_{can}$, is  the Kodaira dimension of $X$.  $X$ is of general type if $\kappa =n$ and in this case, there exists a geometric (singular) K\"ahler-Einstein metric $\omega_{can}$ on $X_{can}$ \cite{So} satisfying
$$Ric(\omega_{can}) = - \omega_{can}.$$  
If $\kappa < n$,  $\pi: X \rightarrow X_{can}$ is  a holomorphic fibration over $X_{can}$, whose generic fibre is a Calabi-Yau manifold. There exists a unique canonical (singular) K\"ahler metric $\omega_{can}$ \cite{ST1, ST2} on $X_{can}$ defined by
\begin{equation}
Ric(\omega_{can}) = - \omega_{can} + \omega_{WP}, 
\end{equation}
where $\omega_{WP}$ is a positive current induced the $L^2$ or Weil-Petersson metric of  Calabi-Yau fibration $\pi: X \rightarrow X_{can}$. 
Furthermore,  $\omega_{can}$ are smooth on $X_{can}^\circ$, away from the critical values of $\pi$.

We now fix a K\"ahler class $\gamma$ and consider the perturbation of $K_X$ by 
\begin{equation}\label{class}
\gamma_t = K_X + t \gamma
\end{equation}
for sufficiently small $t>0$. By the works of \cite{JSS}, there exists $\delta=\delta(\gamma)>0$ such that all $t \in (0, \delta)$, there exists a unique cscK metric $\omega_t \in \gamma_t$.  It is natural to ask what is the asymptotic behavior of this family of cscK metrics when $t\to 0$. The following is a natural extension of the conjecture in \cite{JSS}.

\begin{conjecture} \label{conj1}

The above cscK metric spaces $(X, \omega_t)$ converge to $\overline{(X_{can}^\circ, \omega_{can})}$, the metric completion of $(X_{can}^\circ, \omega_{can})$, in Gromov-Hausdorff topology as $t\rightarrow 0$. Furthermore, $\overline{(X_{can}^\circ, \omega_{can})}$ is homeomorphic to the algebraic variety $X_{can}$.  

\end{conjecture}

When $X$ is of general type, it is proved in \cite{Liu} that $\omega_t$ converges smoothly to $\omega_{can}$ on $X_{can}^\circ$. Furthermore, $(X, g_t)$ have uniformly bounded diameter \cite{GPSS22}. The main goal of this paper is to establish uniform geometric bounds for $(X, g_t)$ for all $X$ with semi-ample $K_X$.

When $X$ is not of general type, i.e., $\kappa<n$, it is much more challenging to obtain both analytic and geometric estimates since the total volume approaches $0$ as $t\to 0$. In particular,  the corresponding cscK metrics must collapse, whereas there is very limited understanding for the behavior and regularity of collapsing canonical K\"ahler metrics.   We would like to point out that concerning the compactness of non-Einstein cscK metrics, the known results \cite{TV, CW} all implicitly require certain non-collapsing conditions and integral control of curvatures, neither of which holds in our study when $\kappa<n$. 

In   \cite{GPSS22, GPSS23},  geometric estimates such as diameter, lower bound of Green's function, Sobolev constants are established under the assumptions of normalized Nash entropy for the Monge-Amp\`ere measures, where collapsing is allowed to take place. Such estimates also lead to a relative volume non-collapsing $$\frac{\vol(B(x,R))}{\vol(X)}\geq c R^\alpha,\quad \forall R\in (0, {\mathrm{diam}(X,\omega)})$$
for some constants $c>0$ and $\alpha>0$, which suffices to conclude the Gromov-Hausdorff compactness in many cases. Our goal is to apply the Sobolev inequality to our study and to establish uniform $L^\infty$-estimates for local potentials and volume measure of the cscK metrics $\omega_t$, which we will also write as $g_t$.  Indeed, we prove the following theorem. 
\begin{theorem}\label{thm:main}
Let $X$ be an $n$-dimensional compact K\"ahler manifold with semiample canonical bundle $K_X$. For any K\"ahler class $\gamma$, there exists $\delta=\delta(\gamma)>0$ such that there exists a unique cscK metric $\omega_t\in K_X+t\gamma$ for $t\in (0, \delta)$ as in \cite{JSS}. Then there exists $\alpha=\alpha(n)>0$ and  $C=C(n, X, \gamma, \delta)$ such that for any $t\in (0,\delta)$, we have
\begin{equation}\label{eqn:diam-noncollapsing}
{\mathrm{diam}}(X, \omega_t) \le C, \quad \frac{\mathrm{Vol}_{\omega_t}(B_{\omega_t}(x, R))}{\mathrm{Vol}_{\omega_t}(X)} \ge C^{-1} R^\alpha,
\end{equation}
for any $R\in (0,1)$, where $B_{\omega_t}(x,R)$ denotes the geodesic ball in $(X,\omega_t)$ with center $x\in X$ and radius $R>0$. Consequently, the family of metric spaces $\{(X, \omega_t)\}_{t\in (0,\delta)}$ is precompact with respect to the Gromov-Hausdorff topology.
\end{theorem}

In fact, we obtain uniform estimates for the Sobolev constant and lower bound of the Green's function associated to $g_t$ as in \cite{GPSS22, GPSS23} due to the uniform estimates in Theorem \ref{thm:tech}. Consequently, for any sequence $\{t_j\}\to 0$, the metric spaces $(X, g_{t_j})$ subsequently converge in Gromov-Hausforff sense to a compact metric space $Z$. It is interesting to investigate the geometry of the limit space $Z$. If Conjecture \ref{conj1} holds, then the twisted K\"ahler-Einstein space $(X_{can}, g_{can})$ arises as the unique geometric limit of cscK metrics in the K\"ahler classes near the canonical $K_X$. This phenomena should be compared to the normalized K\"ahler-Ricci flow on $X$, where the solution converges to $g_{can}$ pointwise on $X_{can}^\circ$ with bounded diameter and scalar curvature \cite{ST3, JS}.

In the next section, we shall prove Theorem \ref{thm:main} assuming a uniform $L^\infty$ estimate of the cscK system. Then in \S\ref{sec:entropy}, we prove a uniform entropy bound based on a uniform $L^\infty$ estimate of $J$-equations. Finally, in \S\ref{sec:Linfty}, we prove the uniform $L^\infty$-estimate (Theorem \ref{thm:tech}) based on the entropy bound in \S\ref{sec:entropy}.

\section{Reduction to uniform a priori estimates for cscK system}

We let $\omega_{FS}$ be the Fubini-Study metric on $\mathbb{CP}^N$ from the pluricanonical map $\pi: X \rightarrow \mathbb{CP}^N$ in (\ref{plmap}) induced by $|mK_X|$ and let %
$$\eta = \frac{1}{m} \pi^*\omega_{FS} \in K_X,$$ which is a semipositive $(1,1)$-form.  
By Yau's theorem \cite{Y}, there is a unique K\"ahler metric $\theta \in \gamma$ such that $\ric(\theta) = -\eta$.   Let $\omega_t$ be the  unique cscK metric $\omega_t \in \gamma_t$ and let
$$\theta_t = \eta + t\theta \in \gamma_t$$  as in (\ref{class}) be the reference metric for each $t\in (0,\delta)$ for fixed $\delta=\delta(\gamma)>0$. Then there exists a unique $\varphi_t$ satisfying
$$\omega_t =\theta_t + \ddbar \varphi_t, ~ \sup_X \varphi_t =0.$$ 
Furthermore, $\varphi_t$ solves the following coupled system
\begin{equation}\label{eqn:cscK}
\left\{\begin{aligned}
&(\theta_t + \ddbar \varphi_t)^n = V_t e^{F_t} \theta^n\\
&\Delta_{\omega_t} F_t =  - \overline{R}_t  - \tr_{\omega_t} (\eta),
\end{aligned}\right.
\end{equation}
where  $ V_t = [\gamma_t]^n= \int_X \theta_t^n$ and $\overline{R}_t = \frac{n [K_X]\cdot [\gamma_t]^n }{[\gamma_t]^n}$. We also have $\int_X e^{F_t} \theta^n  = 1$ from the first equation of  \eqref{eqn:cscK}.

We shall prove that Theorem \ref{thm:main} follows from the following uniform $L^\infty$ estimates.  

\begin{theorem}\label{thm:tech}
There exists a uniform constant $C=C(n, X, \gamma, \delta,  \eta)>0$ such that for all $t\in (0, \delta)$, we have 
\begin{equation}\label{eqn:uniform}
\sup_{t\in (0,\delta)} \left(\| \varphi_t\|_{L^\infty(X)} + \| F_t\|_{L^\infty(X)} \right)\le C.
\end{equation} 
\end{theorem}

We remark that by Theorem \ref{thm:tech}, the metrics $\omega_t$ (after possibly passing to a subsequence) converge weakly  to a positive current with bounded local potentials on $X_{can}$. 

\begin{proof}[Proof of Theorem \ref{thm:main} assuming Theorem \ref{thm:tech}:] We first recall the following results of \cite{GPSS22} for the convenience of readers.

Let $(X,\omega_X)$ be a compact K\"ahler manifold. The $p$-Nash entropy of another K\"ahler form $\omega$ is
$$\eN_{\omega_X,p}(\omega):=\frac{1}{[\omega]^n}\int_X\Big|\log\Big(\frac{1}{[\omega]^n}\frac{\omega^n}{\omega_X^n}\Big)\Big|^p\omega^n.$$
Denote by $K(X)$ the set of K\"ahler metrics on $X$. Consider the class of K\"ahler metrics: 
$$W(\omega_X, A,p,K;\sigma):=\{\omega\in K(X)|\ [\omega]\cdot[\omega_X]^{n-1}\leq A, \eN_{\omega_X,p}(\omega)\leq K, \frac{1}{[\omega]^n}\frac{\omega^n}{\omega_X^n}\geq \sigma\},$$
where $\sigma\geq 0$ is a continuous function. Through Green function's estimates, Guo-Phong-Song-Sturm proved in \cite{GPSS22} that if $dim_H \{\sigma=0\}<2n-1$ and $p>n$, then we can find constants $C=C(\omega_X, n,A,p,K,\sigma)>0$, $c=c(\omega_X, n,A,p,K,\sigma)>0$ and $\alpha=\alpha(n,p)>0$ such that  for any $\omega\in W(\omega_X, A,p,K;\sigma)$, 
$${\mathrm{diam}}(X,\omega)\leq C$$
and for any $x\in X$ and $R\in (0,1]$,
$$\frac{\vol_{\omega}(B_{\omega}(x,R))}{\vol_\omega(X)}\geq cR^\alpha.$$

In our case, note that since
$$\eN_{\theta,p}(\omega_t):=\frac{1}{V_t}\int_X\Big|F_t|^p\omega_t^n,$$
if we have Theorem \ref{thm:tech}, then $N_{\theta,p}(\omega_t)\leq C^p$ for any $p$, and
$$\frac{1}{[\omega_t]^n}\frac{\omega_t^n}{\theta^n}=e^{F_t}\geq e^{-C},$$
where $C$ is the constant in Theorem \ref{thm:tech}.  So we can simply take $\sigma$ to be the constant function $e^{-C}$ and hence the conditions of the main theorem of \cite{GPSS22} are all fulfilled. Consequently, we obtain the desired uniform diameter bound and relative non-collapsing estimate \eqref{eqn:diam-noncollapsing}.

The proof of Gromov-Hausdorff pre-compactness from \eqref{eqn:diam-noncollapsing} is standard: from the relative non-collapsing estimate, for any $\epsilon>0$ sufficiently small, the maximal packing number using disjoint geodesic balls of radius $\epsilon$ is uniformly bounded from above. Then the family is pre-compact by Gromov's pre-compactness theorem (Proposition 5.2 of \cite{Gro}).
\end{proof}

\section{Uniform entropy bounds}\label{sec:entropy}

To prove Theorem \ref{thm:tech}, we first need a $1$-Nash entropy bound. Let $t\in (0,\delta)$ be fixed. All the relevant constants in this section are independent of $t$. 
Let $(\varphi_t, F_t)$ be the solution to the coupled system \eqref{eqn:cscK}. Note that since 
$$\eN_{\theta,1}(\omega_t)=\frac{1}{V_t}\int_X |F_t|\omega_t^n=\int_X |F_t|e^{F_t}\theta^n\leq \int_X F_te^{F_t}\theta^n+\frac{2}{e}[\theta]^n,$$
it suffices to bound $\int_X F_te^{F_t}\theta^n$ by the following proposition.

\begin{prop}\label{prop:2.1}
There is a constant $C>0$ that depends on $n, \theta, \eta$ such that 
\begin{equation}\label{eqn:entropy}
\frac{1}{V_t}\int_X \log \big(\frac{\omega_t^n}{V_t \theta^n} \big) \omega_t^n = \int_X F_t e^{F_t} \theta^n \le C.
\end{equation}
\end{prop}

To prove this upper bound, we start with a family version of  the well-known $\alpha$-invariant argument in \cite{Ti}, based on a local version of \cite{H}:

\begin{lemma}\label{lemma 2.1}
There is a constant $c_0 = c_0(n, \theta, \eta)>0$ such that 
\begin{equation}\label{eqn:2.2}
\frac{1}{V_t}\int_X \log \big(\frac{\omega_t^n}{V_t \theta^n} \big) \omega_t^n \ge 2 c_0 (I_{\theta_t}(\varphi_t) - J_{\theta_t}(\varphi_t)) - C,
\end{equation}
for some uniform constant $C>0$, where $I_{\theta_t}, J_{\theta_t}$ are Aubin's functionals. 
\end{lemma}
\begin{proof}
Note that $\theta_t = \eta + t\theta \le C \theta$ for some uniform constant $C>0$. So, any function in $PSH(X,\theta_t)$ satisfies a uniform $\alpha$-invariant estimate (\cite{Ti, H}): for some $\alpha_0 = \alpha_0(n,X, \eta,\theta)>0$, 
$$\int_X e^{-\alpha_0 \varphi_t} \theta^n \le C.$$
This implies that $$ \frac{1}{V_t} \int_X e^{-\alpha_0 \varphi_t - F_t} \omega_t^n =  \int_X e^{-\alpha_0 \varphi_t - F_t} e^{F_t}\theta^n \le C.$$
Taking log on both sides and applying the Jensen's inequality, we obtain
$$\frac{1}{V_t} \int_X (-\alpha_0 \varphi_t - F_t) \omega_t^n \le \log C.$$
Rearranging the terms gives
\begin{equation}\label{eqn:new3.1}\frac{1}{V_t} \int_X F_t \omega_t^n \ge  \frac{\alpha_0}{V_t} \int_X (-\varphi_t)\omega_t^n - \log C.\end{equation}
The lemma then follows from equivalence of the functionals $I_{\theta_t} - J_{\theta_t}$ and $\frac{1}{V_t} \int_X(-\varphi_t) \omega_t^n$ and choosing $2c_0 = \alpha_0>0$.
\end{proof}

The proof of Proposition \ref{prop:2.1} makes use of the fact that the entropy of $F_t$ is a component of the Mabuchi $K$-energy. We recall the following modified form of Chen-Tian's formula for the $K$-energy: for any $\varphi\in PSH(X, \theta_t)$,
\begin{equation}\label{eqn:K}
K_{\theta_t} (\varphi) =\frac{1}{V_t} \int_X \log \big( \frac{\theta_{t,\varphi}^n}{V_t \theta^n} \big) \theta_{t,\varphi}^n - \frac{1}{V_t} \int_X \log \big( \frac{\theta_t^n}{V_t \theta^n} \big) \theta_t^n + J_{\eta,\theta_t} (\varphi),
\end{equation}
where $\theta_{t,\varphi}=\theta_t+\ddbar\varphi$.
In fact, the usual Chen-Tian formula gives
$$K_{\theta_t} (\varphi) =\frac{1}{V_t} \int_X \log \big( \frac{\theta_{t,\varphi}^n}{\theta_t^n} \big) \theta_{t,\varphi}^n + J_{-Ricc(\theta_t),\theta_t} (\varphi).$$
Since
\begin{align*}
\frac{1}{V_t} \int_X \log \big( \frac{\theta_{t,\varphi}^n}{\theta_t^n} \big) \theta_{t,\varphi}^n & =\frac{1}{V_t} \int_X \log \big( \frac{\theta_{t,\varphi}^n}{V_t \theta^n} \big) \theta_{t,\varphi}^n-\frac{1}{V_t} \int_X \log \big( \frac{\theta_t^n}{V_t \theta^n} \big) \theta_{t,\varphi}^n\\
&= \frac{1}{V_t} \int_X \log \big( \frac{\theta_{t,\varphi}^n}{V_t \theta^n} \big) \theta_{t,\varphi}^n-\frac{1}{V_t} \int_X \log \big( \frac{\theta_t^n}{V_t \theta^n} \big) \theta_t^n\\
&\quad  -\frac{1}{V_t} \int_X \log \big( \frac{\theta_t^n}{V_t \theta^n} \big) (\theta_{t,\varphi}^n-\theta_t^n).
\end{align*}
By writing $\theta_{t,\varphi}^n-\theta_t^n$ as $\int_0^1 \frac{d}{ds}\theta_{t,s\varphi}^n ds$, it is straightforward to check that the last term equals  $J_{\eta+Ricc(\theta_t),\theta_t} (\varphi)$, which in turn implies \eqref{eqn:K}.

It is well-known that cscK metrics are minimizers of $K$-energy, hence we have $K_{\theta_t}(\varphi_t)\le K_{\theta_t}(0) = 0$, which implies that (recall that $\omega_t = \theta_{t,\varphi_t}$)
\begin{align*}
 \frac{1}{V_t} \int_X \log \big( \frac{\theta_t^n}{V_t \theta^n} \big) \theta_t^n & \ge \frac{1}{V_t} \int_X \log \big( \frac{ \omega_{t}^n}{V_t \theta^n} \big) \omega_{t}^n  + J_{\eta,\theta_t}(\varphi_t)\\
\ge &  \frac 12\frac{1}{V_t} \int_X \log \big( \frac{ \omega_{t}^n}{V_t \theta^n} \big) \omega_{t}^n  + J_{\eta + c_0 \theta_t,\theta_t}(\varphi_t) -C,
\end{align*}
where we use Lemma \ref{lemma 2.1} and the following equation of the $J$-functionals
$$c_0 (I_{\theta_t} (\varphi_t) - J_{\theta_t} (\varphi_t) ) + J_{\eta, \theta_t}(\varphi_t) = J_{\eta+ c_0 \theta_t,\theta_t} (\varphi_t).$$ By straightforward calculations, we have $ \frac{1}{V_t} \int_X \log \big( \frac{\theta_t^n}{V_t \theta^n} \big) \theta_t^n \le C$ for some constant $C = C(n,\eta, \theta)>0$. Combining these inequalities, we get
\begin{equation}\label{eqn:2.3}
\frac{1}{V_t} \int_X \log \big( \frac{ \omega_{t}^n}{V_t \theta^n} \big) \omega_{t}^n  + 2 J_{\eta + c_0 \theta_t,\theta_t}(\varphi_t)\le C.
\end{equation}
To get an upper bound of $\frac{1}{V_t} \int_X \log \big( \frac{ \omega_{t}^n}{V_t \theta^n} \big) \omega_{t}^n$, from \eqref{eqn:2.3} we see that it suffices to prove a uniform {\em lower} bound of $J_{\eta + c_0 \theta_t,\theta_t}(\varphi_t)$. \\

In the next step, we will use the existence of minimizer of the $J_{\eta + c_0 \theta_t, \theta_t}$-functional to show the lower bound. For notational convenience, we denote $\chi_t = \eta + c_0 \theta_t$, which satisfies
\begin{equation}\label{eqn:2.5}
c_0 \theta_t \le \chi_t \le (1+c_0) \theta_t.
\end{equation}
By \cite{JSS,SW}, the minimizer $\phi_t\in PSH(X,\theta_t)$ of the $J_{\chi_t, \theta_t}$ exists and solves the $J$-equation
\begin{equation}\label{eqn:J}
(\theta_t + \ddbar \phi_t)^{n-1} \wedge \chi_t = c_t (\theta_t + \ddbar \phi_t)^n,
\end{equation}
where we normalize $\sup_X \phi_t = 0$ and $c_t = c_0 + a_t\ge c_0>0$ with 
\begin{equation}\label{at}
a_t : = \frac{1}{V_t}\int_X \theta_t^{n-1}\wedge \eta.
\end{equation}

We claim that if we can prove a uniform $L^\infty$ bound for the solutions $\phi_t$ of \eqref{eqn:J}, then we will finish the proof of Proposition \ref{prop:2.1}.

In fact, if $\phi_t$ is uniformly bounded, we can obtain a uniform lower bound for $J_{\eta + c_0 \theta_t,\theta_t}(\varphi_t)$ as follows. 
\begin{align*}
J_{\chi_t, \theta_t} (\varphi_t) \ge & ~ J_{\chi_t, \theta_t} (\phi_t)\\
 = &~ \frac{1}{V_t} \int_X \int_0^1 n \phi_t (\chi_t - c_t  \theta_{t,s\phi_t})\wedge (\theta_{t,s\phi_t})^{n-1} ds \ge -C
\end{align*}
for some uniform $C= C(||\phi_t||_{L^\infty(X)})>0$. Consequently, by \eqref{eqn:2.3}, we immediately get a uniform upper bound for $\int_X F_t e^{F_t}\theta^n$.\\

To prove the uniform $L^\infty$ bound for $\phi_t$, we will apply the trick of \cite{Sun} by Moser's iteration. However, we would need the uniform Sobolev inequality from \cite{GPSS23} for the reference metric $\theta_t$.

\begin{lemma}\label{lemma at} There exists $C>0$ such that for all $t\in (0, \delta)$, we have
 $$\left| a_t -  \frac{\kappa}{n} \right| \leq Ct,$$
 for $a_t$ in (\ref{at}).
\end{lemma}
\begin{proof}
By direct computation, it follows that
\begin{align*}
a_t & = ~ \frac{  \sum_{i=0}^{\kappa - 1} t^{n-1-i} \binom{n-1}{i} \int_X\eta^{i+1} \wedge \theta^{n-1-i}  }{ \sum_{i=0}^\kappa t ^{n-i}\binom{n}{i} \int_X\eta^i \wedge \theta^{n-i}    } = \frac{\binom{n-1}{\kappa - 1}}{\binom{n}{\kappa }} + O(t) = \frac{\kappa}{n} + O(t).
\end{align*}
\end{proof}

\begin{lemma}[\cite{JSS}]
There exists a uniform $\delta_0>0$ such that 
\begin{equation}\label{eqn:cone}
n c_t ( \theta_t)^{n-1} - (n-1) \chi_t\wedge (\theta_t) ^{n-2}\ge \delta_0 (\theta_t)^{n-1},
\end{equation}
if $0< t\le \bar t $ for some $\bar t = \bar t (n,\theta,\eta)$ sufficiently small.
\end{lemma}
\begin{proof}
This lemma follows from straightforward calculations as in \cite{JSS}. Indeed, we have
\begin{align*}
&~ n c_t \theta_t^{n-1} - (n-1) \chi_t\wedge\theta_t^{n-2}
\\
= & ~ nc_t (\eta + t\theta)^{n-1} - (n-1) \eta \wedge (\eta + t\theta)^{n-2} - (n-1)c_0\theta_t^{n-1}\\
= & ~ (na_t + c_0) \sum_{i=0}^\kappa \binom{n-1}{i} \eta^i \wedge (t \theta)^{n-1-i} - (n-1) \sum_{i=0}^{\kappa - 1} \binom{n-2}{i} \eta^{i+1}\wedge (t\theta)^{n-2-i}\\
=&~ (n a_t + c_0) (t\theta)^{n-1} + \sum_{i=1}^\kappa A_i \eta^i \wedge (t\theta)^{n-1-i},
\end{align*}
where the coefficients (for $i=1,\ldots, \kappa$)
\begin{align*}
A_i & = (na_t + c_0) \binom{n-1}{i} - (n-1)\binom{n-2}{i-1}\\
\text{(by Lemma \ref{lemma at}) \quad }& = (\kappa +c_0 + O(t))\binom{n-1}{i} -  (n-1)\binom{n-2}{i-1}\\
& = \binom{n-1}{i} (\kappa + c_0 + O(t)  - i)\ge \frac{c_0}{2},
\end{align*}
if $t\le \bar t$ for some sufficiently small $\bar t = \bar t(n, \theta, \eta)>0$. Combining the above inequalities, we finally arrive at
\begin{align*}
&~ n c_t \theta_t^{n-1} - (n-1) \chi_t\wedge\theta_t^{n-2}\\
\ge & \frac{c_0}{2} \sum_{i=0}^\kappa \eta^i\wedge (t\theta)^{n-1-i}\ge \delta_0 \theta_t^{n-1},
\end{align*}
where we may take $$\delta_0 = \frac{c_0}{2 \max_{i=0,\ldots,\kappa}\{\binom{n-1}{i}\}}.$$
\end{proof}

From now on, we additionally impose that $0< t\le \bar t $. For any $s\in [0,1]$, we denote 
$$\theta_{t,s} = \theta_t + \ddbar(s \phi_t), $$
 where $\phi_t$ is the solution to the $J$-equation \eqref{eqn:J}.
\begin{lemma}[\cite{Sun}]
There exists a uniform constant $c_1 = c_1(n,\theta,\eta)>0$ such that 
\begin{equation}\label{eqn:sun}
n c_t (\theta_{t,s}) ^{n-1}  - (n-1) \chi_t \wedge( \theta_{t,s})^{n-2} \ge c_1 (1-s)^{n-1}( \theta_t)^{n-1}. 
\end{equation}

\end{lemma}
\begin{proof}
The proof of this lemma is the same as that of Lemma 2.3 in \cite{Sun}. The point is that the constant $c_1$ here is independent of $t$. For completeness, we include a proof here. We view $\chi_t$ as the reference form in the definition of Hessian operators: for any positive $(1,1)$-form $\theta$
$$\sigma_n(\theta) = \frac{\theta^n}{\chi_t^n},\quad \sigma_{n-1}(\theta) = \frac{n\theta^{n-1}\wedge \chi_t}{\chi_t^n}.$$
We write $\hat \theta_t = \theta_t + \ddbar \phi_t =  \theta_{t,1}$. It is clear that $\theta_{t,s} = s \hat \theta_t + (1-s)\theta_t$. Since for each $i=1,\ldots, n$, the function $$s\mapsto \frac{\sigma_{n-1; i} (\theta_{t,s})}{\sigma_{n-2; i}(\theta_{t,s})}$$ is concave, it follows that
\begin{equation}\label{eqn:sun1}
 \frac{\sigma_{n-1; i} (\theta_{t,s})}{\sigma_{n-2; i}(\theta_{t,s})} \ge s  \frac{\sigma_{n-1; i} (\hat \theta_{t})}{\sigma_{n-2; i}(\hat \theta_{t})} + (1-s)  \frac{\sigma_{n-1; i} (\theta_t)}{\sigma_{n-2; i}(\theta_t)}.
\end{equation}
The first term on the right-hand side of \eqref{eqn:sun1}, $  \frac{\sigma_{n-1; i} (\hat \theta_{t})}{\sigma_{n-2; i}(\hat \theta_{t})} $, is bigger than $\frac{1}{nc_t}$ (see \cite{SW}). By the cone condition \eqref{eqn:cone}, the second term on the right-hand side of \eqref{eqn:sun1}, $\frac{\sigma_{n-1; i} (\theta_t)}{\sigma_{n-2; i}(\theta_t)}$, is no less than 
$$\frac{1}{(1-\frac{\delta_0}{nc_t}) nc_t}\ge (1+\frac{\delta_0}{nc_t}) \frac{1}{nc_t}\ge (1+\bar \delta_0) \frac{1}{nc_t},$$ where $\bar \delta_0 = \frac{\delta_0}{\max_{t\in (0,\bar t]} n c_t}$ is a uniform positive constant. 
Thus the inequality \eqref{eqn:sun1} yields 
\begin{equation}\label{eqn:sun2}
 \frac{\sigma_{n-1; i} (\theta_{t,s})}{\sigma_{n-2; i}(\theta_{t,s})} \ge \frac{s}{nc_t} +\frac{(1 -s)(1+\bar \delta_0)}{ nc_t}. 
 \end{equation}
 In terms of $(n-1,n-1)$-forms, \eqref{eqn:sun2} is equivalent to 
 \begin{equation}\label{eqn:sun3}
 nc_t \theta_{t,s}^{n-1} - (n-1) \chi_t\wedge \theta_{t,s}^{n-2} \ge \bar \delta_0(1-s) (n-1) \chi_t\wedge \theta_{t,s}^{n-2}.
 \end{equation}
On the other hand, since the function $s\mapsto \sigma_{n-1; i}(\theta_{t,s})^{1/(n-1)}$ is concave, we have 
$$ \sigma_{n-1; i}(\theta_{t,s})^{1/(n-1)} \ge s  \sigma_{n-1; i}(\hat \theta_{t})^{1/(n-1)} + (1-s)  \sigma_{n-1; i}(\theta_t)^{1/(n-1)}\ge   (1-s)  \sigma_{n-1; i}(\theta_t)^{1/(n-1)},$$ 
which implies that
$$\chi_t \wedge \theta_{t,s}^{n-2} \ge (1-s)^{n-1} \chi_t \wedge \theta_t^{n-2}.$$
Combining this with \eqref{eqn:sun3} gives that
 \begin{equation}\label{eqn:sun4}
 nc_t \theta_{t,s}^{n-1} - (n-1) \chi_t\wedge \theta_{t,s}^{n-2} \ge \bar \delta_0(1-s)^{n} (n-1) \chi_t\wedge \theta_t^{n-2}\ge c_0\bar  \delta_0 (n-1) (1-s)^n \theta_t^{n-1}. 
  \end{equation}
The lemma is proved with $c_1 =  c_0\bar  \delta_0 (n-1)$.
\end{proof}

We will use the Moser iteration argument to prove the $C^0$ estimates of $\phi_t$. To this end, we need the following uniform Sobolev inequality for the reference metrics $\theta_t$. (Note that by direct computations, the metric $\theta_t$ satisfies the conditions in \cite{GPSS23}, see also their Example 4.1.)

\begin{lemma}[Theorem 2.1 and (4.10) of \cite{GPSS23}]
There exist a constant $q = q(n,X)>1$ and a constant $C=C(n, \theta, \eta)>0$ such that for any $u\in C^1(X)$,
\begin{equation}\label{eqn:Sob}
\Big(\frac{1}{V_t} \int_X |u|^{2q} \theta_t^n\Big)^{1/q} \le \frac{C}{V_t} \int_{X} (u^2 + |\nabla u|_{\theta_t}^2) \theta_t^n.
\end{equation}
\end{lemma}

From these, we can now prove the uniform $L^\infty$ estimate for $\phi_t$, which finishes the proof of Proposition \ref{prop:2.1}.

\begin{prop}
There exists a uniform constant $C>0$ that is independent of $t \in (0,\bar t]$ such that 
$$\sup_X (-\phi_t)\le C.$$
\end{prop}
\begin{proof}
We follow the arguments in \cite{Sun} closely. For any $p>1$, we consider the integral
\begin{equation}\label{eqn:integral}
 \int_X e^{-p \phi_t} \big(  c_t (\theta_{t,\phi_t}^n - \theta_t^n) - \chi_t \wedge ( \theta_{t, \phi_t}^{n-1} - \theta_t^{n-1}  )   \big).
\end{equation}
On one hand, this integral is 
\begin{equation}\label{eqn:integral2}
 \int_X e^{-p \phi_t} \big(  -  c_t   \theta_t^n + \chi_t \wedge  \theta_t^{n-1}     \big)
\le   C \int_X e^{-p\phi_t} \theta_t^n. 
\end{equation}
On the other hand, the integral in \eqref{eqn:integral} is
\begin{align*}
&~ \int_X e^{-p \phi_t}\ddbar \phi_t \wedge \Big( \int_0^1 n c_t \theta_{t,s}^{n-1} - (n-1) \chi_t \wedge \theta_{t,s}^{n-2} ds \Big)\\
= &~ p\int_X e^{-p \phi_t} \sqrt{-1}\partial \phi_t\wedge \bar \partial \phi_t \wedge \Big( \int_0^1 n c_t \theta_{t,s}^{n-1} - (n-1) \chi_t \wedge \theta_{t,s}^{n-2} ds \Big)\\
\ge &~ p\int_X e^{-p \phi_t} \sqrt{-1}\partial \phi_t\wedge \bar \partial \phi_t \wedge \Big( \int_0^1 c_1 (1-s)^{n} ds \theta_t^{n-1}\Big)\\
\ge &~ c_2 p\int_X e^{-p \phi_t} \sqrt{-1}\partial \phi_t\wedge \bar \partial \phi_t \wedge  \theta_t^{n-1},
\end{align*}
for some $c_2>0$ that depends on $n,\theta,\eta$ but is independent of $t$ and $p$. This inequality together with \eqref{eqn:integral2} yield that for some uniform constant $C'>0$ 
\begin{equation}\label{eqn:Moser}
\frac{1}{V_t}\int_X |\nabla e^{-\frac{p}{2} \phi_t}|_{\theta_t}^2 \theta_t^n \le \frac{C' p }{V_t}\int_X e^{-p\phi_t} \theta_t^n.
\end{equation}
Applying the Sobolev inequality \eqref{eqn:Sob} to $u := e^{-p\phi_t/2}$ and using \eqref{eqn:Moser}, we obtain
\begin{equation}\label{eqn:Moser1}
\Big(\frac{1}{V_t} \int_X e^{-q {p\phi_t}{}} \theta_t^n\Big)^{1/q} \le \frac{C p} {V_t} \int_X e^{-p\phi_t} \theta_t^n.
\end{equation}
We now apply the inequality \eqref{eqn:Moser1} with $p_k = q^k$ for $k = 1,2,\ldots$, and \eqref{eqn:Moser1} reads
\begin{equation}\label{eqn:Moser2}
\Big(\frac{1}{V_t} \int_X e^{- {p_{k+1}\phi_t}{}} \theta_t^n\Big)^{1/p_{k+1}} \le C^{1/q} p_k^{1/q}\Big( \frac{1} {V_t} \int_X e^{-p_k\phi_t} \theta_t^n\Big)^{1/p_k}.
\end{equation}
Iterating \eqref{eqn:Moser2} gives
\begin{equation}\label{eqn:Moser3}
\begin{split}
\Big(\frac{1}{V_t} \int_X e^{- {p_{k+1}\phi_t}{}} \theta_t^n\Big)^{1/p_{k+1}} \le & ~C^{\sum_{j=1}^k q^{-j}} q^{\sum_{j=1}^k j q^{-j}}\Big( \frac{1} {V_t} \int_X e^{-q\phi_t} \theta_t^n\Big)^{1/q}\\
\le &~ C  \Big( \frac{1} {V_t} \int_X e^{-q\phi_t} \theta_t^n\Big)^{1/q}.
\end{split}\end{equation}
Letting $k\to\infty$ yields that 
\begin{equation}\label{eqn:final1}
\sup_X e^{-\phi_t}\le C  \Big( \frac{1} {V_t} \int_X e^{-q\phi_t} \theta_t^n\Big)^{1/q}. 
\end{equation}
Finally noting that $\frac{1}{V_t} \theta_t^n \le C \theta^n$ for a uniform constant $C>0$, so 
\begin{align*}  \Big( \frac{1} {V_t} \int_X e^{-q\phi_t} \theta_t^n\Big)^{1/q} \le & ~  C\big( \sup_X e^{-\phi_t}\big)^{\frac{q-\alpha_0}{q}}   \Big(  \int_X e^{-\alpha_0\phi_t}  \theta^n\Big)^{1/q} \\
\text{(by $\alpha$-invariant) }\le & ~  C\big( \sup_X e^{-\phi_t}\big)^{\frac{q-\alpha_0}{q}}, 
\end{align*}
which combined with \eqref{eqn:final1} gives the desired estimate
$$\sup_X e^{-\phi_t}\le C.$$
\end{proof}

\section{From entropy bound to $L^\infty$ estimates}

\label{sec:Linfty}

Given the uniform entropy bound of $F_t$, the $L^\infty$ estimates of $\varphi_t$ and $F_t$ have been proved in \cite{GP}. For completeness, we provide a sketched proof.

Note that by Proposition \ref{prop:2.1}, we have 
\begin{equation}\label{eqn:3.1}\int_X |F_t| e^{F_t} \theta^n \le  C, \end{equation}
and from \eqref{eqn:new3.1} we also have
\begin{equation}\label{eqn:3.2}
\frac{1}{V_t}\int_X (-\varphi_t ) \omega_t^n \le C. 
\end{equation}
Denote $\beta = 1/10$. 
We solve the auxiliary complex Monge-Amp\`ere equations as in \cite{GPT, GP}
$$(\theta_t + \ddbar \psi_k)^n = \frac{\tau_k(-\varphi_t +  \beta F_t)}{A_k} V_t e^{F_t} \theta^n,\quad \sup_X \psi_k = 0,$$
where $\tau_k(x): \mathbb R\to \mathbb R_+$ is a family of positive smooth function that decreases to $x \chi_{\mathbb R_+}(x)$, and $A_k$ is a constant that makes the equation solvable,
$$A_k = \int_X \tau_k(-\varphi_t + \beta F_t) e^{F_t} \theta^n \to \int_\Omega (-\varphi_t + \beta F_t) e^{F_t} \theta^n=: A_\infty,$$
and here $\Omega = \{-\varphi_t + \beta F_t>0\}$. The equations \eqref{eqn:3.1} and \eqref{eqn:3.2} imply that $A_\infty$ is uniformly bounded from above. So we can find a uniform constant $C>0$ such that for any $t>0$, we can find a $k_0$ (possibly depending on $t$) such that $A_k\le C$ for any $k\geq k_0$. In the following, we always assume that $k\geq k_0$.

Consider the test function $$\Psi = -\varepsilon (-\psi_k + \Lambda)^{\frac{n}{n+1}} - \varphi_t + \beta F_t,$$
with the constants chosen such that 
$$\Lambda^{\frac{1}{n+1}} = \frac{2n}{n+1} \varepsilon, \quad \varepsilon = \frac{[(n+1)(n+\beta\overline R_t )]^{n/(n+1)}}{n^{2n/(n+1)}}A_k^{1/(n+1)}.$$
We claim that $\sup_X \Psi\le 0$. If the maximum of of $\Psi$ is obtained at some point in $X\backslash \Omega$, we are done. So assume $\Psi$ take maximum at $x_{\max}\in \Omega$, then at $x_{\max}$ 
\begin{align*}
0\ge & ~ \Delta_{\omega_t} \Psi\\
\ge &~ \frac{n \varepsilon}{n+1} (-\psi_k + \Lambda)^{-\frac{1}{n+1}} \tr_{\omega_t} \theta_{t,\psi_k} -\frac{n \varepsilon}{n+1} (-\psi_k + \Lambda)^{-\frac{1}{n+1}} \tr_{\omega_t} \theta_{t} \\
&~ - n + \tr_{\omega_t} \theta_t -  \beta \overline{R}_t - \beta\tr_{\omega_t} \eta\\
\ge &~ \frac{n^2 \varepsilon}{n+1} (-\psi_k + \Lambda)^{-\frac{1}{n+1}}  \Big(\frac{\theta_{t,\psi_k}^n} {\omega_t^n} \Big)^{\frac{1}{n}} - n -  \beta \overline{R}_t,
 \end{align*}
 by the choice the constants $\varepsilon, \Lambda$. This implies that at $x_{\max}$, $\Psi\le 0$. Hence the claim is proved. Since $\varepsilon\le C$ and $\Lambda\le C$, we obtain 
$$\beta F_t\le -\varphi_t + \beta F_t\le C (-\psi_k + \Lambda)^{\frac{n}{n+1}}.$$ 
Then for any $\epsilon>0$, we can find a constant $C_\epsilon>0$ such that 
 $$\beta F_t\le \epsilon (-\psi_k) + C_\epsilon.$$
 Again, using $\alpha$-invariant, this shows that for any $p>1$
\begin{equation}\label{eqn:Fe}\int_X e^{p F_t} \theta^n \le C_p.\end{equation}
By the family version of Ko\l{}odziej's uniform estimate \cite{K}, developed in \cite{DP} and \cite{EGZ}, we have $\| \varphi_t\|_{L^\infty}\le C$ (see also \cite{GPT}). 

To show the $L^\infty$ estimates of $F_t$, we need the following mean-value inequality in \cite{GPSS22} for the Laplace operator $\Delta_{\omega_t}$.
\begin{lemma}[Lemma 5.1 of \cite{GPSS22}]\label{lemma 2.6}
Under the condition \eqref{eqn:Fe} on $F_t$, there is a uniform constant $C = C(n, p, \theta,\eta)>0$ such that for any $C^2$ function $u$ with $\Delta_{\omega_t} u\ge -a$ for some $a>0 $, the following inequality holds
\begin{equation}\label{eqn:MVI}
\sup_X u\le C\big(a + \frac{1}{V_t} \int_X |u | \omega_t^n  \big).
\end{equation}
\end{lemma}
We first apply Lemma \ref{lemma 2.6} to the function $u: = F_t - \varphi_t$, which satisfies
\begin{align*}\Delta_{\omega_t} u = &~ -\overline{R}_t - \tr_{\omega_t} \eta - n +\tr_{\omega_t}\theta_t \ge -\overline{R}_t - n,
 \end{align*}
and this implies that
$$\sup_X F_t\le \sup_X u \le C\big( \overline{R}_t + n + \int_X (|F_t| + |\varphi_t|   ) e^{F_t} \theta^n \big)\le C.$$
To get the lower bound of $F_t$, we apply Lemma \ref{lemma 2.6} to the function $u: = - F_t$, which fulfills the equation
$$\Delta_{\omega_t} u = -\Delta_{\omega_t} F_t = \overline{R}_t + \tr_{\omega_t} \eta \ge -\overline{R}_t,$$
and we obtain
$$\sup_X (-F_t) \le C\big(|\overline{R}_t| + \int_X |F_t| e^{F_t} \theta^n   \big)\le C,$$
thus the lower bound of $F_t$ follows, and we finish the proof of Theorem \ref{thm:tech}.

\bigskip



\end{document}